\documentclass[11pt,reqno]{amsart}
\usepackage{amsfonts,amssymb,amsmath,amsthm}
\usepackage[colorlinks,citecolor=blue,urlcolor=black,linkcolor=black]{hyperref}

\newtheorem{thm}{Theorem}[section]
\newtheorem{mainthm}{Theorem}

\newtheorem{lemma}[thm]{Lemma}
\newtheorem{claim}{Claim}[thm]

\theoremstyle{definition}
\newtheorem{defn}[thm]{Definition}
\newtheorem{q}{Question}
\newtheorem{notation}[thm]{Notation}

\theoremstyle{remark}
\newtheorem{remark}[thm]{Remark}

\newcommand\symdiff{\mathbin{\bigtriangleup}}
\renewcommand{\mid}{\mathrel{|}\allowbreak}

\DeclareMathOperator{\crit}{crit}
\DeclareMathOperator{\cf}{cf}

\newcommand\s{\subseteq}
\newcommand*\axiomfont[1]{\textsf{\textup{#1}}}
\newcommand\zfc{\axiomfont{ZFC}}
\newcommand\gch{\axiomfont{GCH}}

\begin{document}
\title[GCH above a strongly compact cardinal]{The generalized continuum hypothesis above a strongly compact cardinal}
\author{Zhixing You}
\address{Einstein Institute of Mathematics, The Hebrew University of Jerusalem, Israel}
\email{zhixingy121@gmail.com}

\date{Preprint as of October 5, 2026.}

\begin{abstract}
    We answer a long-standing open question of Woodin
    by proving that if $\kappa$ is strongly compact and $\gch$ holds below $\kappa$, then $\gch$ holds.
\end{abstract}
\maketitle

\section{Introduction}
Cardinal arithmetic studies the possible values of cardinal exponentiation.
Of particular interest is the continuum function $\lambda\mapsto 2^\lambda$.
The \emph{Generalized Continuum Hypothesis} ($\gch$) asserts that this function
takes its least possible value at every infinite cardinal:
$2^\lambda=\lambda^+$.

Easton's theorem \cite[Theorem~10.7]{Kanamori} shows that $\zfc$ imposes essentially no constraint
on the continuum function $\lambda\mapsto 2^\lambda$ at regular cardinals
beyond monotonicity and K\"onig's inequality $\cf(2^\lambda)>\lambda$.

Large cardinal axioms can impose additional restrictions on the
continuum function. Strong compactness provides an important example.
Let $\kappa$ be a strongly compact cardinal.
Solovay's classical theorem \cite[p.~308]{Kanamori} establishes
$\gch$ at every singular strong limit cardinal above $\kappa$.

For regular cardinals, since every strongly compact cardinal is measurable, Scott's theorem
shows that $\gch$ below $\kappa$ already implies
$2^\kappa=\kappa^+$ \cite[Corollary~5.8]{Kanamori}.
Under the stronger assumption that $\kappa$ is supercompact,
$\gch$ below $\kappa$ implies $\gch$
\cite[Proposition~22.2]{Kanamori}.
Woodin~\cite{Apter2000} asked whether strong compactness alone
suffices for this conclusion.
\begin{q}[Woodin, {\cite{Apter2000}}]\label{qwoodin}
    Suppose that $\kappa$ is strongly compact and $\gch$ holds
    below $\kappa$. Must $\gch$ hold?
\end{q}

In this paper we answer the question in the affirmative.

\begin{mainthm}\label{thma}
    If $\kappa$ is a strongly compact cardinal and $\gch$ holds below $\kappa$,
    then $\gch$ holds.
\end{mainthm}

\section{Preliminaries}\label{sec:prelim}

Our notation and terminology are standard and follow \cite{Kanamori}.
We work in a transitive model $V$ of $\zfc$.
For an infinite cardinal $\lambda$, let
$2^{<\lambda}:=\sup\{2^\mu\mid \mu<\lambda\text{ is a cardinal}\}$.
We say that \emph{$\gch$ holds below $\kappa$} iff $2^\mu=\mu^+$ for every infinite cardinal $\mu<\kappa$.
For a set $X$, let $|X|$ denote the cardinality of $X$.
For sets $X$ and $Y$, $X\symdiff Y$ denotes their symmetric difference.
For an elementary embedding $j:V\rightarrow M$,
let $\crit(j)$ denote the critical point of $j$, i.e., the least ordinal moved by $j$.
We write $j[X]:=\{j(x)\mid x\in X\}$ for any set $X$.
Superscripts $M$ indicate computations in $M$; for $X\in M$, $|X|^M$ and $\mathcal P(X)^M$ denote the cardinality and power set of $X$ computed in $M$, respectively.
Unless indicated otherwise, cardinalities and power sets are computed in $V$.

\begin{defn}[Strong compactness]\label{dfine}
    Suppose that $\kappa$ is an uncountable cardinal.
    \begin{enumerate}
        \item For any $\lambda\ge\kappa$, $\kappa$ is \emph{$\lambda$-strongly compact} iff there exists a definable elementary embedding $j:V\rightarrow M$ into a transitive class $M$ with critical point $\kappa$,
              such that there exists a $D\in M$ with $j[\lambda]\subseteq D$ and $M\models |D|<j(\kappa)$.
              In this case, we say that $D$ is \emph{minimal} iff $|C|^M\ge|D|^M$ for every $C\in M$ with $j[\lambda]\subseteq C$.
        \item We say that $\kappa$ is \emph{strongly compact} iff it is $\lambda$-strongly compact for every $\lambda\ge\kappa$.
    \end{enumerate}
\end{defn}

\begin{notation}
    In the following, without loss of generality,
    we may further assume that $D\subseteq \sup(j[\lambda])$ by taking $D \cap \sup(j[\lambda])$ if necessary.
\end{notation}

The next lemma is standard, but we include it for completeness.
\begin{lemma}\label{lpullback}
    Suppose $\kappa\leq \lambda$ are two uncountable cardinals
    and $\kappa$ is $\lambda$-strongly compact,
    as witnessed by a definable elementary embedding $j:V\rightarrow M$ and a minimal $D$. Let $d:=|D|^M$. Then:
    \begin{enumerate}
        \item[(1)] If $S\in M$ and $|S|^M< d$, then $|\{\alpha<\lambda\mid j(\alpha)\in S\}|<\lambda$.
        \item[(2)] There is a set $R\subseteq d$ cofinal in $d$ with $|R|\le\lambda$.
    \end{enumerate}
\end{lemma}
\begin{proof}
    (1) Let $E:=\{\alpha<\lambda\mid j(\alpha)\in S\}$.
    Suppose for contradiction that $|E|=\lambda$.
    Fix a bijection $e:\lambda\rightarrow E$.
    By elementarity, $j(e):j(\lambda)\rightarrow j(E)$ is a bijection in $M$.
    Let
    $$C:=\{\beta<j(\lambda)\mid j(e)(\beta)\in S\}.$$
    Then $C\in M$. Since $j(e)$ maps $C$ injectively into $S$,
    $|C|^M\le|S|^M<d$.

    For every $\alpha<\lambda$, we have $e(\alpha)\in E$, so
    $$j(e)(j(\alpha))=j(e(\alpha))\in S.$$
    Thus $j(\alpha)\in C$, and hence $j[\lambda]\subseteq C$.
    Note also that $C\in M$ and $|C|^M<d$.
    This contradicts the minimality of $D$.

    (2) Fix in $M$ a bijection $f:D\rightarrow d$, and put $R:=f[j[\lambda]]$.
    Then $R\subseteq d$ and $|R|=\lambda$.
    If $R$ were bounded in $d$, there would be some $\gamma<d$ with $R\subseteq\gamma$.
    Then $f^{-1}[\gamma]\in M$ would contain $j[\lambda]$ and have $M$-cardinality
    $|\gamma|^M<d$, contradicting the minimality of $D$.
    Hence $R$ is cofinal in $d$.
\end{proof}

\section{Proof of the main theorem}\label{sec:count}

The following lemma is the key to the proof of Theorem~\ref{thma}.
\begin{lemma}\label{lcapacity}
    Suppose $\kappa\le\lambda$ are two uncountable cardinals
    and $\kappa$ is $\lambda$-strongly compact,
    as witnessed by a definable elementary embedding $j:V\rightarrow M$ and a minimal $D$.
    Put $d:=|D|^M$ and $F:=\{j(A)\cap D\mid A\subseteq\lambda\}$.
    For every $T\in M$ with $|T|^M\le d$, we have
    \begin{equation}\label{eq1}
        |F \cap T|\le 2^{<\lambda}.
    \end{equation}
\end{lemma}
\begin{proof}
    we first reduce to $|T|^M<d$ and then find a set $E\subseteq j[\lambda]$ of size less than $\lambda$ that separates the members of $F\cap T$.

    By Lemma~\ref{lpullback}(2), fix a set $R\subseteq d$ cofinal in $d$ with $|R|\le\lambda$.
    Fix in $M$ an injection $q:T\rightarrow d$, and for every $r\in R$ let $T_r:=\{t\in T\mid q(t)\le r\}$.
    Then $T_r\in M$ and $|T_r|^M\le|r+1|^M<d$, since $d$ is an infinite cardinal of $M$.
    By cofinality of $R$, $T=\bigcup_{r\in R}T_r$.
    Thus, if the inequality~\eqref{eq1} holds for each $T_r$,
    it also holds for $T$,
    since $|R|\leq \lambda$ and
    $\lambda\cdot 2^{<\lambda}=2^{<\lambda}$.

    We may therefore assume that $|T|^M<d$.

    \begin{claim}\label{cldelta}
        Suppose that $A,B\s\lambda$ are distinct. Then
        $$\min\bigl((j(A)\cap D)\symdiff(j(B)\cap D)\bigr)=j\bigl(\min(A\symdiff B)\bigr).$$
    \end{claim}
    \begin{proof}
        Let $\delta:=\min(A\symdiff B)$.
        We need to prove that $\min\bigl((j(A)\cap D)\symdiff(j(B)\cap D)\bigr)=j(\delta)$.
        First, we have
        $$(j(A)\cap D)\symdiff(j(B)\cap D)=\bigl(j(A)\symdiff j(B)\bigr)\cap D.$$
        Second, by elementarity, $j(A)\symdiff j(B)=j(A\symdiff B)$, and its least element is $j(\delta)$.
        Since $j[\lambda]\s D$, it follows that $j(\delta)\in D$ and
        hence
        $$j(\delta)= \min\bigl((j(A)\symdiff j(B))\cap D\bigr)= \min\bigl((j(A)\cap D)\symdiff(j(B)\cap D)\bigr).$$
    \end{proof}

    Let
    $S:=\{\min(X\symdiff Y)\mid X,Y\in T\cap\mathcal P(D)^M\text{ and }X\neq Y\}$ and let $E:=j[\lambda]\cap S$.
    Since $D$ is a set of ordinals, the set $S$ is well-defined.
    Since $T\cap\mathcal P(D)^M\in M$, $S$ belongs to $M$.
    Moreover, $F\cap T\subseteq T\cap\mathcal P(D)^M$, and
    $|S|^M\le|T\times T|^M<d$.
    Since $j$ is injective and $j[\{\alpha<\lambda\mid j(\alpha)\in S\}]=E$, Lemma~\ref{lpullback}(1) gives
    $$|E|=|\{\alpha<\lambda\mid j(\alpha)\in S\}|<\lambda.$$

    Now define $\Phi:F\cap T\longrightarrow\mathcal P(E)$ by letting
    $\Phi(X):=X\cap E$ for every $X\in F\cap T$.
    We claim that $\Phi$ is injective.
    Take distinct $X,Y\in F\cap T$, and choose $A,B\subseteq\lambda$ such that
    $X=j(A)\cap D$ and $Y=j(B)\cap D$.
    Then $A\neq B$. Let $\delta:=\min(A\symdiff B)$.
    By Claim~\ref{cldelta},
    $$j(\delta)=\min(X\symdiff Y)\in S.$$
    Thus $j(\delta)\in j[\lambda]\cap S=E$, and $j(\delta)$ belongs to exactly one of $X$ and $Y$.
    Consequently $j(\delta)\in\Phi(X)\symdiff\Phi(Y)$, so $\Phi(X)\neq\Phi(Y)$.
    Therefore,
    $|F\cap T|\le|\mathcal P(E)|=2^{|E|}\le 2^{<\lambda}$.
\end{proof}

We now prove Theorem~\ref{thma}.

\begin{proof}[Proof of Theorem~\ref{thma}]
    We argue by induction on the infinite cardinals $\lambda \geq \kappa$.
    Assume that $\gch$ holds below $\lambda$.
    Then $2^{<\lambda}=\lambda$ and
    it suffices to prove that $2^\lambda\le (2^{<\lambda})^+$.
    Choose a definable elementary embedding $j:V\rightarrow M$ and a minimal $D\in M$ witnessing the $\lambda$-strong compactness of $\kappa$.
    Let $d:=|D|^M<j(\kappa)$.

    Since $\gch$ holds below $\kappa$, elementarity gives
    $M\models$ ``$\gch$ holds below $j(\kappa)$''.
    Since $d$ is an infinite cardinal of $M$ below $j(\kappa)$,
    we have $M\models|\mathcal P(D)|=d^+$.
    Let $\theta:=(d^+)^M$.
    Fix in $M$ a bijection $h:\theta\rightarrow\mathcal P(D)^M$, and put
    $$F:=\{j(A)\cap D\mid A\subseteq\lambda\}.$$
    Since $j[\lambda]\subseteq D$, for every $A\subseteq\lambda$ and $\alpha<\lambda$,
    $$\alpha\in A\quad\Longleftrightarrow\quad j(\alpha)\in j(A)\cap D.$$
    Thus $A\mapsto j(A)\cap D$ is injective, and $|F|=2^\lambda$.

    For every $\beta<\theta$, we have $h[\beta]\in M$ and $|h[\beta]|^M\le d$, because $\theta=(d^+)^M$.
    Now applied Lemma~\ref{lcapacity} to $T:=h[\beta]$, therefore gives
    $$|F\cap h[\beta]|\le 2^{<\lambda}.$$
    Hence $2^\lambda=|F|=|F\cap h[\theta]|\le (2^{<\lambda})^+$.
\end{proof}

\begin{remark}
    The same argument applies locally:
    if $\kappa\le \lambda$ are two uncountable cardinals and $\kappa$ is $\lambda$-strongly compact and $\gch$ holds below $\kappa$,
    then $2^\mu=\mu^+$ for every infinite cardinal $\mu\le\lambda$,
    since $\lambda$-strong compactness implies $\mu$-strong compactness whenever $\kappa\le\mu\le\lambda$.
\end{remark}

\begin{remark}
    The proofs above don't use $\crit(j)=\kappa$.
    If $\kappa$ is $\omega_1$-strongly compact (in the sense of Definition~\ref{dfine}, with
    $\crit(j)\ge\omega_1$ in place of $\crit(j)=\kappa$), the same proofs apply.
    Consequently, if $\kappa$ is $\omega_1$-strongly compact and $\gch$ holds below $\kappa$, then $\gch$ holds.
\end{remark}

\section*{Acknowledgements}

This work used AI assistance from GPT-6-Astra with the \texttt{ultra} reasoning setting,
with GPT-6.1-Sol and GPT-6-Astra invoked dynamically during the investigation.
The author takes
full responsibility for the paper.

The author was supported by the Israel Science Foundation, Grants no. 1967/21 and 3469/25.

\end{document}